\documentclass[11pt]{article}
\usepackage{amsfonts,epsf,amsmath,amssymb}
\usepackage{epsfig}
\usepackage{latexsym}
\usepackage{pgf,tikz}
\usepackage{bm}
\usepackage{lineno}
\usepackage{enumitem}

\newtheorem{theorem}{\bf Theorem}[section]
\newtheorem{corollary}[theorem]{\bf Corollary}
\newtheorem{lemma}[theorem]{\bf Lemma}
\newtheorem{proposition}[theorem]{\bf Proposition}

\newtheorem{remark}[theorem]{\bf Remark}

\newcommand{\qed}{\hfill $\square$ \bigskip}
\newcommand{\Bin}{{\cal B}}
\newcommand{\Fib}{{\cal F}}
\newcommand{\Luc}{{\cal L}}

\begin{document}
\title{Edges in Fibonacci cubes, Lucas cubes and complements}

\author{
Michel Mollard\footnote{Institut Fourier, CNRS, Universit\'e Grenoble Alpes, France email: michel.mollard@univ-grenoble-alpes.fr}
}
\date{\today}
\maketitle

\begin{abstract}
\noindent The {\em Fibonacci cube} of dimension $n$, denoted as $\Gamma_n$, is the subgraph of the hypercube induced by vertices with no consecutive 1's. The irregularity of a graph $G$ is the sum of $|d(x)-d(y)|$ over all edges $\{x,y\}$ of $G$. In two recent paper based on the recursive structure of $\Gamma_n$ it is proved that the irregularity  of $\Gamma_n$ and $\Lambda_n$ are two times the number of edges of $\Gamma_{n-1}$ and $2n$ times the number of vertices of $\Gamma_{n-4}$, respectively. Using an interpretation of the irregularity in terms of couples of incident edges of a special kind (Figure \ref{irr}) we give a bijective proof of both results. For these two graphs we deduce also a constant time algorithm for computing the imbalance of an edge. In the last section using the same approach we determine the number of edges and the sequence of degrees of the cube complement of $\Gamma_n$.
\end{abstract}

\noindent
{\bf Keywords:} Irregularity of graph, Fibonacci cube, Lucas cube, cube-complement, daisy cube. 

\noindent
{\bf AMS Subj. Class. }: 05C07,05C35

\section{Introduction and notations}

An interconnection topology can be represented by a graph $G=(V,E)$, where $V$ denotes the processors and $E$ the communication links.
The hypercube $Q_n$ is a popular interconnection network because of its structural properties.\\
The {\em Fibonacci cube} of dimension $n$, denoted as $\Gamma_n$, is the subgraph of the hypercube induced by vertices with no consecutive 1's. This graph was introduced in \cite{hsu93} as a new interconnection network.

$\Gamma_n$ is an isometric subgraph of the hypercube which is inspired in the Fibonacci numbers. It has attractive recurrent structures such as
its decomposition into two subgraphs which are also Fibonacci cubes by themselves.  Structural properties of these graphs were more extensively 
studied afterwards. See \cite{survey} for a survey. \\
\indent Lucas cubes, introduced in \cite{mupe2001}, have attracted the attention as well due to the fact that these cubes are 
the cyclic version of Fibonacci cubes. They have also been widely studied
\cite{ca2011, camo2012, dedo2002, km2012, klmope2011, ra2013}. \\
The determination of degree sequence \cite{klmope2011} is one of the first enumerative results about Fibonacci cubes.

Let $G=(V(G),E(G))$ be a connected graph. The degree of a vertex $x$ is denoted by $d_G(x)$ or $d(x)$ when there is no ambiguity.
The \emph{imbalance} of an edge $e=\{x,y\}\in E(G)$ is defined by $imb_G(e)=|d_G(x)-d_G(y)|$.
The \emph{irregularity} of a non regular graph $G$ is $$irr(G)=\sum_{e\in E(G)}{imb_G(e)}.$$
This concept of irregularity was introduced in \cite{albert1997} as a measure of graph's global non-regularity.

In two recent papers \cite{ali2019, egec2020} using the inductive structure of Fibonacci cubes it is proved that $irr(\Gamma_n)=2|E(\Gamma_{n-1})|$ and $irr(\Lambda_n)=2n|V(\Gamma_{n-4})|$. One of our motivation is to give direct bijective proofs of these remarkable properties.

The \emph{generalized Fibonacci cube} $\Gamma_n(s)$ is the graph obtained from $Q_n$  by removing all vertices that contain a given binary string $s$  as a substring. For example  $\Gamma_n(11)=\Gamma_n$. \emph{Daisy cubes} are  an other kind of generalization of Fibonacci cubes introduced in \cite{kamo2018}. \\
For $G$  an induced subgraph of $Q_n$, the \emph{cube-complement of $G$} is the graph induced by the vertices of $Q_n$ which are not in $G$.  In \cite{ve2019} the questions whether the cube complement of  generalized Fibonacci cube is connected, an isometric subgraph of a hypercube or a median graph are studied. It is also proved in the same paper that the cube-complement of a daisy cube is a daisy cube. We consider in the last section $\overline{\Gamma}_n $ the cube complement of $\Gamma_n$.

We give the number of edges of $\overline{\Gamma}_n$ and determine, using the main lemma of the first section, the degree sequence of $\overline{\Gamma}_n $. We will also study the embedding of $\Gamma_n$ in $\overline{\Gamma}_n$.

We will next give  some concepts and notations needed in this paper.
We note by $[1,n]$ the set of integers $i$ such that $1\leq i \leq n$.
The vertex set of the \emph{hypercube of dimension $n$} $Q_n$ is the set $\Bin_n$ of  binary strings of length $n$, two vertices being adjacent if they differ in precisely one position. We will note $\overline{x_i}$ the binary complement of $x_i$.

Let $x=x_1\ldots x_n$ be a binary string and $i\in[1,n]$ we will denote by $x+\delta_i$ the string $x'_1\ldots x_n'$ where $x'_j=\overline{x_j}$ for $j=i$ and $x'_j=x_j$ otherwise. We will say that the edge $\{x,x+\delta_i\}$ \emph{uses the direction $i$}. The endpoint $x$ such that $x_i=1$ of an edge using the direction $i$  will be called  \emph{upper endpoint} and $y$ the \emph{lower endpoint}.

A {\em Fibonacci string} of length $n$ is a binary string $b=b_1 b_2\ldots b_n$ with $b_i\cdot b_{i+1}=0$ for $1\leq i<n$. In other words a Fibonacci string is a binary string without $11$ as substring.\\ 
The {\em Fibonacci cube} $\Gamma_n$ ($n\geq 1$) is the subgraph of $Q_n$ induced by the Fibonacci strings of length $n$. Because of the empty string $\epsilon$, $\Gamma_0 = K_1$. 

A Fibonacci string $b$ of length $n$ is a \textit{Lucas string} if $b_1 \, \cdotp b_n \neq 1$. 
That is, a Lucas string has no two consecutive 1's including the first and the last elements of the string. 
The \textit{Lucas cube} $\Lambda_n$ is the subgraph of $Q_n$ induced by the Lucas strings of length $n$. 
We have $\Lambda_0 = \Lambda_1 = K_1$.

\begin{figure}
  \centering
\setlength{\unitlength}{1 mm}
\begin{picture}(160, 30)

\newsavebox{\gtwo}
\savebox{\gtwo}
  (30,30)[bl]
  {
  
\put(10,10){\circle*{2}}
\put(10,20){\circle*{2}}
\put(10,30){\circle*{2}}

\put(10,10){\line(0,1){10}}
\put(10,20){\line(0,1){10}}

\put(5,10){$01$}
\put(5,20){$00$}
\put(5,30){$10$}

  }

 \newsavebox{\gthreeb}
\savebox{\gthreeb}
  (30,30)[bl]
  {
  
\put(10,10){\circle*{2}}
\put(10,20){\circle*{2}}
\put(10,30){\circle*{2}}

\put(10,10){\line(0,1){10}}
\put(10,20){\line(0,1){10}}

\put(3,10){$011$}
\put(3,20){$111$}
\put(3,30){$110$}

  }
 \newsavebox{\gthree}
\savebox{\gthree}
  (40,30)[bl]
  {
  
\put(10,10){\circle*{2}}
\put(10,20){\circle*{2}}
\put(10,30){\circle*{2}}

\put(20,20){\circle*{2}}
\put(20,10){\circle*{2}}

\put(10,20){\line(1,0){10}}
\put(10,10){\line(0,1){10}}
\put(10,20){\line(0,1){10}} 
\put(10,10){\line(1,0){10}}
\put(20,10){\line(0,1){10}}

\put(3,10){$001$}
\put(3,20){$000$}
\put(3,30){$010$}
\put(21,20){$100$}
\put(21,10){$101$}

 }
\newsavebox{\lthree}
\savebox{\lthree}
  (40,30)[bl]
  {
  
\put(10,10){\circle*{2}}
\put(10,20){\circle*{2}}
\put(10,30){\circle*{2}}

\put(20,20){\circle*{2}}

\put(10,20){\line(1,0){10}}
\put(10,10){\line(0,1){10}}
\put(10,20){\line(0,1){10}} 

\put(3,10){$001$}
\put(3,20){$000$}
\put(3,30){$010$}
\put(21,20){$100$}

 }

\newsavebox{\gfour}
\savebox{\gfour}
  (40,30)[bl]
  {

\put(10,20){\circle*{2}}
\put(10,30){\circle*{2}}
\put(20,10){\circle*{2}}
\put(20,20){\circle*{2}}
\put(20,30){\circle*{2}}
\put(30,10){\circle*{2}}
\put(30,20){\circle*{2}}
\put(10,10){\circle*{2}}

\put(10,20){\line(1,0){10}}
\put(10,30){\line(1,0){10}}

\put(10,20){\line(0,1){10}}

\put(20,10){\line(1,0){10}}
\put(20,20){\line(1,0){10}}

\put(20,10){\line(0,1){10}}
\put(20,20){\line(0,1){10}}
\put(30,10){\line(0,1){10}} 

\put(10,10){\line(1,0){10}}
\put(10,10){\line(0,1){10}}

\put(20,12){$0001$}
\put(20,22){$0000$}
\put(20,32){$0010$}
\put(31,20){$0100$}
\put(1,20){$1000$}
\put(1,30){$1010$}
\put(31,10){$0101$}

\put(1,10){$1001$}

 }

\newsavebox{\gfourb}
\savebox{\gfourb}
  (40,30)[bl]
  {

\put(10,20){\circle*{2}}
\put(10,30){\circle*{2}}
\put(20,10){\circle*{2}}
\put(20,20){\circle*{2}}
\put(20,30){\circle*{2}}
\put(30,10){\circle*{2}}
\put(30,20){\circle*{2}}
\put(10,10){\circle*{2}}

\put(10,20){\line(1,0){10}}
\put(10,30){\line(1,0){10}}

\put(10,20){\line(0,1){10}}

\put(20,10){\line(1,0){10}}
\put(20,20){\line(1,0){10}}

\put(20,10){\line(0,1){10}}
\put(20,20){\line(0,1){10}}
\put(30,10){\line(0,1){10}} 

\put(10,10){\line(1,0){10}}
\put(10,10){\line(0,1){10}}

\put(20,12){$0111$}
\put(20,22){$1111$}
\put(20,32){$1101$}
\put(31,20){$1011$}
\put(1,20){$1110$}
\put(1,30){$1100$}
\put(31,10){$0011$}

\put(1,10){$0110$}

 }

\put(-8,-5){\usebox{\gtwo}}  
\put(9,-5){\usebox{\gthree}} 
\put(35,-5){\usebox{\lthree}}  

\put(63,-5){\usebox{\gfour}}
\put(102,-5){\usebox{\gthreeb}}  
\put(120,-5){\usebox{\gfourb}}
\end{picture}
\caption{$\Gamma_2=\Lambda_2$, $\Gamma_3$, $\Lambda_3$, $\Gamma_4$, $\overline{\Gamma}_3$ and $\overline{\Gamma}_4$.}\label{F1}

\end{figure}
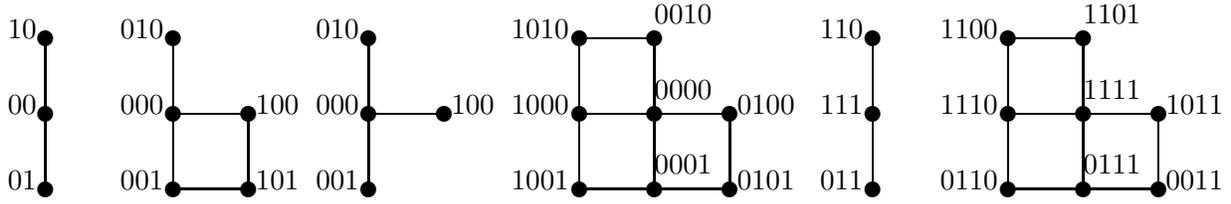

Let $F_n$ be the $n$th Fibonacci number: $F_0=0$, $F_1=1$, $F_n=F_{n-1}+F_{n-2}$ for $n\geq2$.

Let $\Fib_n$ and $\Luc_n$ be the sets of strings of Fibonacci strings and Lucas strings of length $n$.
Let $\Fib_n^{1.}$ and $\Fib_n^{0.}$ be the set of strings of $\Fib_n$ that begin with $1$ and that do not begin with $1$, respectively. Note that with this definition $\Fib_0^{0.}=\{\epsilon\}$ and $\Fib_0^{1.}=\emptyset$. Let $\Fib_n^{.0}$ be the set of strings of $\Fib_n$ that do not end with $1$. Thus $|\Fib_n^{.0}|=|\Fib_n^{0.}|$.
Let $\Fib_n^{00}$ be the set of strings of $\Fib_n^{0.}$ that  do not end with $1$. With this definition  $\Fib_0^{00}=\{\epsilon\}$, $\Fib_1^{00}=\{0\}$ and $\Fib_2^{00}=\{00\}$.

From $\Fib_{n+2}=\{0\bm{s}; \bm{s}\in \Fib_{n+1}\}\cup\{10\bm{s}; \bm{s}\in \Fib_{n}\},\Fib_{n+1}^{0.}=\{0\bm{s}; \bm{s}\in \Fib_{n}\} \text{and }\Fib_{n+1}^{1.}=\{1\bm{s}; \bm{s}\in \Fib_{n}^{0.}\}$ we obtain the following classical result.

\begin{proposition}\label{nbstring}
Let $n \ge 0$. The numbers of Fibonacci strings in $\Fib_{n}$, $\Fib_{n}^{0.}$ and $\Fib_{n}^{1.}$ are  $|\Fib_{n}|=F_{n+2}$, $|\Fib_{n}^{0.}|=F_{n+1}$ and $|\Fib_{n}^{1.}|=F_{n}$ respectively.
Let $n \ge 1$. The number of Fibonacci strings in  $\Fib_n^{00}$ is  $|\Fib_n^{00}|=F_{n}$.
\end{proposition}

The following expressions for the number of edges in  $\Gamma_{n}$ are obtained in \cite{Kla2005} and \cite{mupe2001} .
\begin{proposition}\label{nbedges}
Let $n \ge 0$. The number of edges in $\Gamma_{n}$ is $|E(\Gamma_{n})|=\sum_{i=1}^n{F_{i}F_{n-i+1}}=\frac{nF_{n+1}+2(n+1)F_n}{5}$ and satisfies  the induction formula $|E(\Gamma_{n+2})|=|E(\Gamma_{n+1})|+|E(\Gamma_{n})|+|V(\Gamma_{n})|$ .
\end{proposition}
\begin{remark}
Let $\{x,x+\delta_i\}$ be an edge and  $\theta(x)=((x_{1}x_{2}\ldots x_{i-1}), (x_{i+1} x_{i+2}\ldots x_{n})) $. 
A combinatorial interpretation of  $|E(\Gamma_{n})|=\sum_{i=1}^n{F_{i}F_{n-i+1}}$ is that for any $i\in[1,n]$ $\theta$ is a one to one mapping between the set of edges using the direction $i$ and the Cartesian product $\Fib_{i-1}^{.0}\times \Fib_{n-i}^{0.}$
\end{remark}

Let $G$ be an induced subgraph of $Q_n$. Let $e=\{x,y\}$ be an edge of $G$ where $y$ is the lower endpoint of $e$ and $x=y+\delta_i$. An edge $e'=\{y,y+\delta_j\}$ of $G$ will be called an \emph{imbalanced} edge for $e$ if $x+\delta_j \notin V(G)$ and thus $\{x,x+\delta_j\}\notin E(G)$. Note that such couple of edges does not exist for $G= Q_n$. We will prove in the next to sections that for $G=\Gamma_n$ and $G=\Lambda_n$ the irregularity of $G$ is the number of such couples of edges (Figure \ref{irr}).

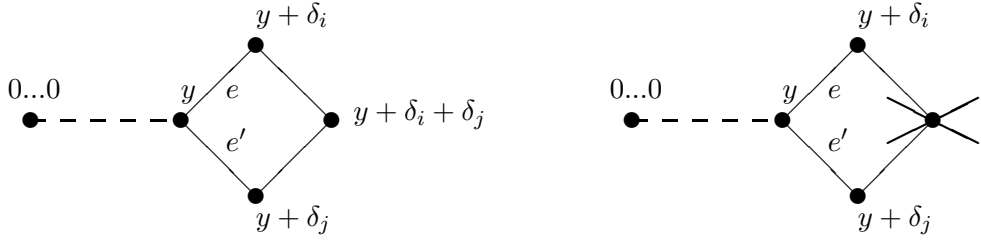
\begin{figure}
  \centering
\setlength{\unitlength}{1 mm}
\begin{picture}(160, 30)

\newsavebox{\leftcerise}
\savebox{\leftcerise}
  (80,30)[bl]
  {
  
\put(10,30){\circle*{2}}
\put(30,30){\circle*{2}}
\put(50,30){\circle*{2}}
\put(40,40){\circle*{2}}
\put(40,20){\circle*{2}}

\put(10,30){\line(1,0){2}}
\put(14,30){\line(1,0){2}}
\put(18,30){\line(1,0){2}}
\put(22,30){\line(1,0){2}}
\put(26,30){\line(1,0){2}}

\put(30,30){\line(1,1){10}}
\put(30,30){\line(1,-1){10}}
\put(50,30){\line(-1,1){10}}
\put(50,30){\line(-1,-1){10}}

\put(7,33){$0...0$}
\put(30,33){$y$}
\put(40,43){$y+\delta_i$}
\put(36,33){$e$}
\put(40,16){$y+\delta_j$}
\put(36,26){$e'$}
\put(53,30){$y+\delta_i+\delta_j$}

  }

 \newsavebox{\rightcerise }
\savebox{\rightcerise}
  (80,30)[bl]
  {
  
\put(10,30){\circle*{2}}
\put(30,30){\circle*{2}}
\put(50,30){\circle*{2}}
\put(40,40){\circle*{2}}
\put(40,20){\circle*{2}}

\put(10,30){\line(1,0){2}}
\put(14,30){\line(1,0){2}}
\put(18,30){\line(1,0){2}}
\put(22,30){\line(1,0){2}}
\put(26,30){\line(1,0){2}}

\put(30,30){\line(1,1){10}}
\put(30,30){\line(1,-1){10}}
\put(50,30){\line(-1,1){10}}
\put(50,30){\line(-1,-1){10}}

\put(7,33){$0...0$}
\put(30,33){$y$}
\put(40,43){$y+\delta_i$}
\put(36,33){$e$}
\put(40,16){$y+\delta_j$}
\put(36,26){$e'$}

\thicklines
\put(44,33){\line(2,-1){12}}
\put(44,27){\line(2,1){12}}
  }

\put(0,-5){\usebox{\leftcerise}}  
\put(80,-5){\usebox{\rightcerise}} 
\end{picture}
\caption{$irr(\Gamma_n)$ and $irr(\Lambda_n)$ count the couples of edges $(e,e')$ of the right kind.}
\label{irr}
\end{figure}

\section{Edges in Fibonacci cube}
\begin{lemma}\label{lemmazero}
Let $x,y$ be two strings in $\Fib_n$ with $y=x+\delta_i$ and $x_i=1$. Then for all $j\in[1,n]$ we have
 $$ x+\delta_j\in \Fib_n \text { implies }y+\delta_j\in \Fib_n.$$
 \end{lemma}
\begin{proof}
 Assume  $y+\delta_j\notin \Fib_n$ then   $y_k=1$ for some $k$ in  $\{j-1,j+1\}\cap [1,n]$. But for all $ p\in[1,n]$ $x_p=0$ implies  $y_p=0$. Thus $x_k=1$ and   $x+\delta_j\notin \Fib_n$.
\end{proof}\qed

\begin{lemma}\label{lemmadelta}
Let $x,y$ two strings in $\Fib_n$ with $y=x+\delta_i$.
Then for all $j\in[1,n]$ with $|i-j|>1$ we have $$x+\delta_j\in \Fib_n\text { if and only if }y+\delta_j\in \Fib_n.$$
 \end{lemma}
\begin{proof}
\begin{itemize}
\item If $x_j=1$  then $y_j=x_j=1$ and both $x+\delta_j$ and $y+\delta_j$ belong to $\Fib_n$.
\item Assume  $x_j=0$ thus $y_j=0$. We have
$$x+\delta_j\in\Fib_n\text{ if and only if } x_{k}=0 \text{ for all } k\in \{j-1,j+1\}\cap[1,n]$$ and
$$y+\delta_j\in\Fib_n\text{ if and only if } y_{k}=0 \text{ for all } k\in \{j-1,j+1\}\cap[1,n].$$
But $i\notin\{j-1,j+1\}\cap[1,n]$ thus $x_k=y_k$ for all $k$ in this set and the two conditions are equivalent.
\end{itemize}
\end{proof} \qed

\begin{corollary}\label{corbij}
Let $n\geq2$ then $irr(\Gamma_n)$ is the number of couples $(e,e')\in E(\Gamma_n)^2$ where $e'$ is an imbalanced edge for $e$.
\end{corollary}
\begin{proof}
By Lemma \ref{lemmazero} if $e=\{x,y\}$ is an edge using the direction $i$ with upper endpoint $x$ then $d(y)\geq d(x)$ and  $imb(e)$ is the number of imbalanced edges for $e$. The conclusion follows.
\end{proof}\qed

Furthermore assume that $e=\{x,y\}$ uses the direction $i$ with $x_i=1$ and let $e'=\{y,y+\delta_j\}$ be an imbalanced edge for $e$. Then by Lemma \ref{lemmadelta} we have $j=i+1$ or $j=i-1$.

We will call $e'$ a \emph{right }or \emph{left imbalanced} edge for $e$ accordingly. Let $R_{\Gamma_n}$ and $L_{\Gamma_n}$ be the sets of couples $(e,e')$ where $e'$ is a right imbalanced edge for $e$ and a left imbalanced edge for $e$, respectively, where $e$ goes through $E(\Gamma_n$).

\begin{theorem}
Let $n\geq2$. There exists a one to one mapping between $R_{\Gamma_n}$ or $L_{\Gamma_n}$ and $E(\Gamma_{n-1})$.
\end{theorem}
\begin{proof}
Let $(e,e')\in R_{\Gamma_n}$. Assume that $x$ is the upper endpoint of $e=\{x,y\}$. We have thus $y=x+\delta_i$ and $x_i=1$ for some $i\in[1,n-1]$.

Let $\theta((e,e'))= \{x_1 x_2 \dots x_{i-1}1x_{i+2}x_{i+3}\dots x_{n},x_1 x_2 \dots x_{i-1}0x_{i+2}x_{i+3}\dots x_{n}\}$.
Since $x$ and $y$ belong to $\Fib_n$ and the edges $e$, $e'$ use the direction $i$, $i+1$ we have $x_k=y_k=0 $ for $k$ in $ \{i-1,i+2\}\cap[1,n]$. Therefore $x_1 x_2 \dots x_{i-1}1x_{i+2}x_{i+3}\dots x_{n}$ is a Fibonacci string and $\theta((e,e'))$ belongs to $E(\Gamma_{n-1})$.\\
Conversely let $f=\{z_1 z_2 \dots z_{i-1}0z_{i+1}z_{i+2}\dots z_{n-1},z_1 z_2 \dots z_{i-1}1z_{i+1}z_{i+2}\dots z_{n-1}\}$ be an arbitrary edge of $\Gamma_{n-1}$ then $z_k=0 $ for $k$ in $ \{i-1,i+1\}\cap[1,n-1]$. Thus $x=z_1 z_2 \dots z_{i-1}10z_{i+1}z_{i+2}\dots z_{n-1}$ and $t=z_1 z_2 \dots z_{i-1}01z_{i+1}z_{i+2}\dots z_{n-1}$ are in $\Fib_n$. The edge $\{t,t+\delta_{i+1}\}$ is a right imbalanced edge for the edge $\{x,x+\delta_{i}\}$. Furthermore $\theta(\{x,x+\delta_i\},\{t,t+\delta_{i+1}\})=f$ and $\theta$ is a bijection.

Similarly let $\phi((e,e'))= \{x_1 x_2 \dots x_{i-2}1x_{i+1}x_{i+2}\dots x_{n},x_1 x_2 \dots x_{i-2}0x_{i+1}x_{i+2}\dots x_{n}\}$ where $x$ is the upper end point of an edge $e$ using the direction $i$ and such that $(e,e')\in L_{\Gamma_n}$. Then $\phi$ is a one to one mapping between $L_{\Gamma_n}$ and $E(\Gamma_{n-1})$.
\end{proof}\qed

As an immediate  corollary we deduce the result of Alizadeh and his co-authors \cite{ali2019} 
\begin{corollary}
$$ irr(\Gamma_n)=2|E(\Gamma_{n-1})|.$$
\end{corollary}
An other consequence of Lemma \ref{lemmadelta}  is the following classification of the edges according to their imbalance. Note that from this classification we obtain a constant time algorithm for computing the imbalance of an edge of $\Gamma_n$. 
\begin{theorem}
Let $n\geq4$ and $e=\{x,y\}$ be an edge of $\Gamma_n$ using direction $i$. Then  $imb(\{x,y\})$ follows Table \ref{tab:Table1}.
\end{theorem}
\begin{table}[ht]
	\centering
		\begin{tabular}{l|c|c|c c|c|r}
		$imb(\{x,x+\delta_i\})$& $i=1$ & $i=2$ &\multicolumn{2}{|c|}{$3\leq i \leq n-2$}  & $i=n-1$ & $i=n$\\ \hline
		& $x_3$ & $x_4$ & $x_{i-2}$&$x_{i+2}$ & $x_{n-3}$ & $x_{n-2}$\\ \hline
		0 & 1 & & 1 & 1 & & 1 \\ \hline
		1 & 0 & 1& $\begin{matrix} 0 \\ 1 \end{matrix}$ & $\begin{matrix} 1 \\0 \end{matrix}$ & 1 & 0\\ \hline
		2 && 0 & 0 & 0 &0& \\ \hline
			
		\end{tabular}
	\caption{$imb(e)$ in $\Gamma_n$ }
	\label{tab:Table1}
\end{table}
\begin{proof}
Assume that $x$ is the upper endpoint of the edge $e=\{x,y\}$.
There exists an edge $e'$ such that $e'$ is a right imbalanced edge for $e$ if and only if $i\in[1,n-1]$ and $e'=\{y,y+\delta_{i+1}\}$ thus if $y+\delta_{i+1}\in\Fib_n$. Since $y_i=0$,  $y+\delta_{i+1}$ is a Fibonacci string  if and only if $i=n-1$  or if $y_{i+2}=x_{i+2}=0$ in the general case $i\in[1,n-2]$.

Similarly there exists an edge $e'$ such that $e'$ is a left imbalanced edge for $e$ if and only if $i\in[2,n]$ and $e'=\{y,y+\delta_{i-1}\}$ thus if $y+\delta_{i-1}\in\Fib_n$. Since $y_i=0$,  $y+\delta_{i-1}$ is a Fibonacci string  if and only if $i=2$  or if $y_{i-2}=x_{i-2}=0$ when $i\in[3,n]$.

Therefore $imb(e)$ is completely determined by the values of $x_{i+2},x_{{i-2}}$ according to Table \ref{tab:Table1}.
\end{proof}\qed\\
Let $e$ be an edge of $\Gamma_n$ then by Lemma \ref{lemmadelta}  $imb(e)\leq2$.
Let $A$, $B$, $C$ be the sets of edges with $imb(e)=0$, $imb(e)=1$ and $imb(e)=2$, respectively.
\begin{theorem}\label{thfib}
Let $n\geq2$. The numbers of edges of $\Gamma_n$ with imbalance 0,1 and 2 are respectively
$$|A|=\sum_{i=3}^{n-2}{F_{i-2}F_{n-i-1}}+2F_{n-2}$$
$$|B|=2\sum_{i=1}^{n-3}{F_{i}F_{n-i-2}}+2F_{n-1}$$
$$|C|=\sum_{i=2}^{n-1}{F_{i-1}F_{n-i}}.$$
\end{theorem}
\begin{remark}
Note that  $|B|+2|C|=2|E(\Gamma_{n-1})|$ and we obtain again the result of Alizadeh and his co-authors.
\end{remark}
\begin{proof}
The case $n\leq3$ is obtained by direct inspection.\\
Assume $n\geq4$.\\
For $i\in[1,n]$ let $E_i$ be the set of edges $\{x,y\}$ of $\Gamma_i$ with $y=x+\delta_i$. Let $A_i=A\cap E_i$, $B_i=B\cap E_i$ and $C_i=C\cap E_i$.
Let $e=\{x,y\}$ be an edge of $\Gamma_n$.
\begin{itemize}
\item If $e\in A_i$ then by Table \ref{tab:Table1} we have  $i\in [3,n-2]$ or $i\in\{1,n\}$.
If $i\in [3,n-2]$ then $\theta(e)=(x_{1}x_{2}\ldots x_{i-2}, x_{i+2} x_{i+3} \ldots x_{n}) $ is a one to one mapping between $A_i$ and $\Fib_{i-2}^{.1}\times \Fib_{n-i-1}^{1.}$.\\
If $i=1$  then $\phi(e)=x_{3} x_{4} \ldots x_{n} $ is a one to one mapping between  $A_1$ and $\Fib_{n-2}^{1.}$. Similarly $\Psi(e)=x_{1} x_{2} \ldots x_{n-2} $ is a one to one mapping between  $A_n$ and $\Fib_{n-2}^{.1}$. By Proposition \ref{nbstring} we obtain $|A|=\sum_{i=3}^{n-2}{F_{i-2}F_{n-i-1}}+2F_{n-2}$.

\item If $e\in C_i$ then by Table \ref{tab:Table1} we have  $i\in [2,n-1]$.
Let $\theta(e)=(x_{1}x_{2}\ldots x_{i-2},x_{i+2} x_{i+3} \ldots x_{n}) $. Then $\theta$  is a one to one mapping between $C_i$ and $\Fib_{i-2}^{.0}\times \Fib_{n-i-1}^{0.}$. The expression of $|C|$ follows.

\item Assume $e\in B_i$ and that there exists a right imbalanced edge for $e$ therefore no left imbalanced edge. We have thus $i\in[1,n-1]$ and $i\neq 2$.
If $i\in [3,n-1]$ then $\theta(e)=(x_{1}x_{2}\ldots x_{i-2}, x_{i+2} x_{i+3} \ldots x_{n}) $ is a one to one mapping this kind of edges and  and $\Fib_{i-2}^{.1}\times \Fib_{n-i-1}^{0.}$.\\
If $i=1$ then $\phi(e)=x_{3} x_{4} \ldots x_{n} $ is a one to one mapping between this kind of edges and $\Fib_{n-2}^{0.}$. Thus this case contributes $\sum_{i=3}^{n-1}{F_{i-2}F_{n-i}}+F_{n-1}=\sum_{i=1}^{n-3}{F_{i}F_{n-i-2}}+F_{n-1}$ to $B$.

\item Assume $e\in B_i$ and that there exists a left imbalanced edge for $e$ thus  no right imbalanced edge. By a similar construction this case contributes also $\sum_{i=1}^{n-3}{F_{i}F_{n-i-2}}+F_{n-1}$ to $B$. The expression of $|B|$ follows.

\end{itemize}

\end {proof}

\section{Edges in Lucas cube}

For any integer $i$ let $\bm{i}=((i-1)\mod n) +1$. Thus $\bm{i}=i$ for $i\in[1,n]$ and $\bm{n+1}=1$, $\bm{0}=n$. With this notation $i$ and  $\bm{i+1}$ are cyclically consecutive in $[1,n]$. Therefore for $x\in\Luc_n$ with $x_i=0$  the string $x+\delta_i$ belongs to $\Luc_n $ if and only if $x_k=0$ for all $k\in\{\bm{i-1},\bm{i+1}\}$. Note also that $k\in\{\bm{i-1},\bm{i+1}\}$ if and only if $i\in\{\bm{k-1},\bm{k+1}\}$
\begin{lemma}\label{lemmazerol}
Let $x,y$ be two strings in $\Luc_n$ with $y=x+\delta_i$ and $x_i=1$. Then for all $j\in[1,n]$ we have
 $$ x+\delta_j\in \Luc_n \text { implies }y+\delta_j\in \Luc_n.$$
 \end{lemma}
\begin{proof}
 Assume  $y+\delta_j\notin \Luc_n$ then   $y_k=1$ for some $k$ in  $\{\bm{j-1},\bm{j+1}\}$. But for all $ p\in[1,n]$ $x_p=0$ implies  $y_p=0$. Thus $x_k=1$ and   $x+\delta_j\notin \Luc_n$.
\end{proof}\qed

\begin{lemma}\label{lemmadeltal}
Let $x,y$ two strings in $\Luc_n$ with $y=x+\delta_i$.
Then for all $j\in[1,n]$ with $j\notin\{\bm{i-1},\bm{i+1}\}$ we have $$x+\delta_j\in \Luc_n\text { if and only if }y+\delta_j\in \Luc_n.$$
 \end{lemma}
\begin{proof}
This is true for $j=i$ thus assume $j\neq i$.
\begin{itemize}
\item If $x_j=1$  then $y_j=x_j=1$ and both $x+\delta_j$ and $y+\delta_j$ belong to $\Luc_n$.
\item Assume  $x_j=0$ thus $y_j=0$. We have
$$x+\delta_j\in\Luc_n\text{ if and only if } x_{k}=0 \text{ for all } k\in \{\bm{j-1},\bm{j+1}\}$$ and
$$y+\delta_j\in\Luc_n\text{ if and only if } y_{k}=0 \text{ for all } k\in \{\bm{j-1},\bm{j+1}\}.$$
But $i\notin\{\bm{j-1},\bm{j+1}\}$ thus $x_k=y_k$ for all $k$ in this set and the two conditions are equivalent.
\end{itemize}
\end{proof} \qed

From this two lemmas we deduce the equivalent for Lucas cube of Corollary \ref{corbij}. 
\begin{corollary}
Let $n\geq2$ then $irr(\Lambda_n)$ is the number of couples $(e,e')\in E(\Lambda_n)^2$ where $e'$ is an imbalanced edge for $e$.
\end{corollary}

Let $e'=\{y,y+\delta_j\}$ be an imbalanced edge for $e$ then by Lemma \ref{lemmadeltal} we have $j=\bm{i+1}$ or $j=\bm{i-1}$. We will call $e'$ a cyclically right or cyclically left imbalanced edge for $e$ accordingly. Let $R^i_{\Lambda_n}$ be the set of $(e,e')$ where $e'$ is a cyclically right imbalanced edge for $e$ and $e$ uses the direction $i$. Similarly let $L^i_{\Lambda_n}$ be the equivalent set for cyclically left imbalanced edges. 

\begin{theorem}
Let $n\geq4$ and $i\in[1,n]$. There exists a one to one mapping between $R^i_{\Lambda_n}$ or $L^i_{\lambda_n}$ and $\Fib_{n-4}$.
\end{theorem}
\begin{proof}
Since $x_1 x_2\dots x_n\mapsto x_i x_{i+1}\dots x_n x_1 x_2 \dots x_{i-1}$ is an automorphism of $\Lambda_n$ we can assume without loss of generality that $i=1$. 
Let $(e,e')$ in $R^1_{\Lambda_n}$. Assume that $x$ is the upper endpoint of $e=\{x,y\}$. We have thus $y=x+\delta_1$ and $x_1=1$. Let $\theta((e,e'))= x_4 x_5 \dots x_{n-1}$. As a substring of $x$ the string $x_4 x_5 \dots x_{n-1}$ belongs to $\Fib_{n-4}$. Furthermore since $e$ and $e'$ use the directions $1$ and $2$ we have $x_n=x_2=x_3=0$. Therefore  $\theta((e,e'))= x_4 x_5 \dots x_{n-1}$ defines  $x$ thus defines $(e,e')$  and $\theta$ is injective.    
Conversely let $z_1 z_2 \dots z_{n-4}$ be an arbitrary string of $\Fib_{n-4}$. Let $x=100z_1 z_2\dots z_{n-4}0$,  $t=010z_1 z_2\dots z_{n-4}0$, $e=\{x,x+\delta_1\}$ and $e'=\{t,t+\delta_2\}$. Note that $t+\delta_2=x+\delta_1$ and $x+\delta_2\notin \Luc_n$ thus by Lemma \ref{lemmadeltal} $(e,e')\in R^1_{\Lambda_n}$. Therefore $\theta$  is surjective.
The proof that $\phi((e,e'))= x_3 x_4 \dots x_{n-2}$ where $e=\{x,x+\delta_1\}$ defines a one to one mapping between $L^1_{\lambda_n}$ and $\Fib_{n-4}$ is similar. 
\end{proof}\qed

As an immediate  corollary we deduce the result obtained in \cite{egec2020} 
\begin{corollary}
For all $n\geq 3$ $irr(\Lambda_n)= 2n|F_{n-2}|$.
\end{corollary}

Like in $\Gamma_n$ it is not necessary to know the degree of his endpoints for computing the imbalance of an edge in $\Lambda_n$.
\begin{theorem}
Let $n\geq 4$ and $e=\{x,y\}$ be an edge of $\Lambda_n$ with $y= x+\delta_i$. Then  $imb(e)$ follows Table \ref{tab:Table2} where their indices $\bm{i-2}$ and $\bm{i+2}$ are taken cyclically in $[1,n]$.
\begin{table}[ht]
	\centering
		\begin{tabular}{l|c|r}
		$imb(\{x,x+\delta_i\})$& $x_{\bm{i-2}} $ & $x_{\bm{i+2}}$\\ \hline
		0 & 1 &  1  \\ \hline
		1 & $\begin{matrix} 0 \\ 1 \end{matrix}$ & $\begin{matrix} 1 \\0 \end{matrix}$ \\ \hline
		2 & 0 & 0\\ \hline
			
		\end{tabular}
	\caption{$imb(e)$ in $\Lambda_n$ }
	\label{tab:Table2}
\end{table}
\end{theorem}

\begin{proof}
Assume that $x$ is the upper endpoint of the edge. Since $x_{\bm{i+1}}=y_{\bm{i+1}}=0$ there exists a couple $(e,e')$ in $R^i_{\Lambda_n}$ if and only if $e'=\{y,y+\delta_{\bm{i+1}}\}$ and $x_{\bm{i+2}}=0$. Since $x_{\bm{i-1}}=y_{\bm{i-1}}=0$ there exists a couple $(e,e')$ in $L^i_{\Lambda_n}$ if and only if $e'=\{y,y+\delta_{\bm{i-1}}\}$ and $x_{\bm{i-2}}=0$. Therefore $imb(e)$ is completely determined by the values of $x_{\bm{i+2}},x_{\bm{i-2}}$ according to Table \ref{tab:Table2}.
\end{proof}\qed

Let $e=\{x,x+\delta_i\}$ be an edge of $\Lambda_n$ and  $\theta(e)=x_{i+1} x_{i+2}\ldots x_{n} x_{1} x_{2}\ldots x_{i-1} $. Note that $\theta$ is a one to one mapping between the set of edges using the direction $i$ and $\Fib_{n-1}^{00}$. This remark gives a combinatorial interpretation of the well known result $|E(\Lambda_n)|=nF_{n-1}$ \cite{mupe2001}. We will use the same idea for the number edges with a given imbalance.

\begin{corollary}
Let $n\geq5$ then the imbalance of any edge $e=\{x,y\}$ in $ \Lambda_n$ is at most 2. Furthermore if $A$, $B$ and $C$ are the sets of edges with imbalance 0,1 and 2 respectively then $|A|=nF_{n-5}$, $|B|=2nF_{n-4}$ and $|C|=nF_{n-3}$.
\end{corollary}
\begin{proof}
For $i\in[1,n]$ let $E_i$ be the set of edges $\{x,y\}$ using direction $i$. Since the number of edges in $E_i$  with  a given imbalance is independent of $i$ we can assume without loss of generality that $i=1$ and consider  $A_1=A\cap E_1$, $B_1=B\cap E_1$ and $C_1=C\cap C_1$. Let $x$ be the end point such that $x_1=1$. We have thus $x_2=x_n=0$, $x+\delta_2\notin \Luc_n $ and $x+\delta_n\notin \Luc_n $.
\begin{itemize}
\item Assume $x_3=x_{n-1}=0$. Then $y+\delta_2\in \Luc_n  $, $y+\delta_n\in \Luc_n$ and the edge $\{x,y\}$ belongs to $C_1$. Furthermore 
$\theta(x)= x_{3} x_{4}\ldots x_{n-1}$ is one to one mapping between the set of this kind of edges and $\Fib_{n-3}^{00}$. The contribution of this case to $C_1$ is  $F_{n-3}$.
\item Assume $x_3=x_{n-1}=1$. Then $y+\delta_2\notin \Luc_n  $, $y+\delta_n\notin \Luc_n$ and the edge $\{x,y\}$ belongs to $A_1$.  Since $x_4=x_{n-2}=0, $\  
$\theta(x)= x_{4} x_{5}\ldots x_{n-2}$ is one to one mapping between the set of this kind of edges and $\Fib_{n-5}^{00}$. The contribution of this case to $A_1$ is  $F_{n-5}.$
\item Assume $x_3=1$ and $x_{n-1}=0$. Then $y+\delta_2\notin \Luc_n  $, $y+\delta_n\in \Luc_n$. The edge $\{x,y\}$ belongs to $B_1$.  Furthermore $\theta(x)= x_{4} x_{5}\ldots x_{n-1}$ is one to one mapping between the set of this kind of edges and $\Fib_{n-4}^{00}$. The contribution of this case to $B_1$ is  $F_{n-4}$.
\item The case  $x_3=0$ and $x_{n-1}=1$ is similar and thus contributes also $F_{n-4}$ to $B_1$ 
\end{itemize}

\end{proof}\qed\\

\section{Cube-complement of Fibonacci cube}

Let $\overline{\Fib}_n$ be the set of binary strings of length $n$  with $11$ as substring. We will call the strings in $\overline{\Fib}_n$ non-Fibonacci strings of length $n$. The cube complement of $\Gamma_n$ is $\overline{\Gamma}_n$ the subgraph of $Q_n$ induced by $\overline{\Fib}_n$.\\
Note that $\overline{\Fib}_n$ is connected since there is always a path between any vertex $x\in V({\overline{\Gamma}_n})$ and $1^n$.
Furthermore  $|V({\overline{\Gamma}_n})|=2^n-F_{n+2}$.\\
Let $A_n$,$B_n$,$C_n$ be the sets of edges of $Q_n$ incident to 0,1 and 2, respectively, strings of $\Fib_n$. We have thus  $A_n=E({\overline{\Gamma}_n})$ and $C_n=E(\Gamma_n)$.
\begin{proposition}\label{nbtotal}
$|E(Q_n)|=|E(\overline{\Gamma}_n)|+|B_n|+|E(\Gamma_n)|$ is  the total number of $0$'s in binary strings of length $n$.\\
$|B_n|+|E(\Gamma_n)|$ is  the total number of $0$'s in Fibonacci strings of length $n$.\\
$|B_n|+|E(\overline{\Gamma}_n)|$ is  the total number of $1$'s in non-Fibonacci strings of length $n$.\\
$|E(\overline{\Gamma}_n)|$ is the total number of $0$'s in non-Fibonacci strings  of length $n$.\\
$|E(\Gamma_n)|$ is the total number of $1$'s in Fibonacci strings  of length $n$.\\ 

\end{proposition} 
\begin{proof}
Let $e$ be an edge of $Q_n$ and let $x,y$ such that $e=\{x,y\}$ with $x_i=0$ and $y_i=1$. Define the mappings $\phi(\{x,y\})= (x,i)$ and $\psi(\{x,y\})= (y,i)$. Note that $\phi$ is a one to one mapping between  $E(Q_n)$ and $\{(s,i); s\in\Bin_n,s_i=0\}$ the set of  $0$'s appearing in strings of $\Bin_n$. Likewise  $\psi$ is a one to one mapping between  $E(Q_n)$ and $\{(s,i); s\in\Bin_n,s_i=1\}$ the set of $1$'s appearing in strings of $\Bin_n$.\\
Furthermore an edge incident to exactly one Fibonacci string $z$ is mapped by $\phi$ to $(z,i)$ and by $\psi$ to $(z+\delta_i,i)$.
Therefore the restriction of the reverse mapping $\phi^{-1}$ to $\{(s,i); s\in\Fib_n,s_i=0\}$ the set of $0$'s appearing in strings of $\Fib_n$, is a one to one mapping to the edges of $E(\Gamma_n)\cup B_n$. 
For the same reason the restriction of the reverse mapping $\psi^{-1}$ to  $\{(s,i); s\in\Fib_n,s_i=1\}$, the set of $1$'s appearing in strings of $\Fib_n$, is a one to one mapping to the edges of $E (\overline{\Gamma}_n)\cup B_n$.\\ 
Since a binary string is a Fibonacci string or a non-Fibonacci string we can deduce the last two affirmations from the previous. We can also give a direct proof. Indeed the restriction of $\phi$ to edges of $\overline{\Gamma}_n$ define a one to one mapping between $E(\overline{\Gamma}_n)$ and $\{(s,i); s\in\overline{\Fib}_n,s_i=0\}$. Likewise the restriction of $\psi$ to edges of $\Gamma_n$ define a one to one mapping between $E(\Gamma_n)$ and $\{(s,i); s\in\Fib_n,s_i=1\}$.
\end{proof}\qed 

\begin{proposition}\label{nbtotal0}
The total number of number of $0$'s in Fibonacci strings of length $n$ is $\sum_{i=1}^{n}{F_{i+1}F_{n-i+2}}$.
\end{proposition} 
\begin{proof}
Let $s$ be a Fibonacci string of length $n$ and $i\in[1,n]$ then $s_1s_2\dots s_{i-1}$ and $s_{i+1}s_{i+2}\dots s_{n}$ are Fibonacci strings. Reciprocally if $u$ and $v$ are Fibonacci strings then $u0v$ is also a Fibonacci string.   Therefore the mapping define by $\theta(s,i)=(s_{1}s_{2}\dots s_{i-1}, s_{i+1}s_{i+2}\dots s_{n}) $ is  a one to one mapping between $\{(s,i); s\in\Fib_n,s_i=0 \}$ and the Cartesian product $\Fib_{i-1}\times \Fib_{n-i}$. The identity follows.
\end{proof}\qed

\begin{theorem}
The number of edges of $\overline{\Gamma}_n$ is given by the equivalent expressions:
\begin{enumerate}[label=(\roman*)]
\item $|E(\overline{\Gamma}_n)|=n2^{n-1}-\sum_{i=1}^{n}{F_{i+1}F_{n-i+2}}.$
\item $|E(\overline{\Gamma}_n)|=n2^{n-1}-\frac{4nF_{n+1}+(3n-2)F_n}{5}.$
 \end{enumerate}
\end{theorem}

\begin{proof}
Combining the first two identities in Proposition \ref{nbtotal} together with Proposition \ref{nbtotal0} we obtain the first expression.\\
For the second expression note first that the $n$ edges of $Q_n$ incident to a vertex of $\Fib_n$ belongs to $E(\Gamma_n)$ or $B_n$. Making the sum over all vertices of $\Fib_n$ the edges of $E(\Gamma_n)$ are obtained two times therefore $nF_{n+2}=|B_n|+2|E(\Gamma_n)|$. By Proposition \ref{nbtotal} $|E(\overline{\Gamma}_n)|=|E(Q_n)|-|B_n|-|E(\Gamma_n)|=n2^{n-1}-nF_{n+2}+|E(\Gamma_n)|$. Using the expression of $|E(\Gamma_n)|$ given by Proposition \ref{nbedges} we obtain the final result.
\end{proof}\qed

The sequence  $(|E(\overline{\Gamma}_n)|,n\geq1)=0,0,2,10,35,104,\dots $ can also be obtain by an inductive relation:\\
\begin{proposition}
The number of edges of $\overline{\Gamma}_n$ is the sequence defined by \\
$|E(\overline{\Gamma}_n)|=|E(\overline{\Gamma}_{n-1})|+|E(\overline{\Gamma}_{n-2})|+(n+4)2^{n-3}-F_{n+2} \ \ (n \geq 3)$\\
$|E(\overline{\Gamma}_1)|=|E(\overline{\Gamma}_2)|=0.$
\end{proposition}
\begin{proof}
Let $n\geq 3$. 
Let $\overline{\Fib_n}^{1.}$ be the set of strings of $\overline{\Fib}_n$ that begin with $1$. Since $\overline{\Fib_n}^{1.}=\{10\bm{s}; \bm{s}\in\overline{\Fib}_{n-2}\}\cup\{11\bm{s}; \bm{s}\in \Bin_{n-2}\}$ we have $|\overline{\Fib_n}^{1.}|=2^{n-1}-F_n$. This identity is also valid for $n=1$ or $n=2$.\\
Consider the following partition of the set of vertices of $\overline{\Gamma}_n$:
$\overline{\Fib}_{n}=\{0\bm{s}; \bm{s}\in \overline{\Fib}_{n-1}\}\cup\{10\bm{s}; \bm{s}\in \overline{\Fib}_{n-2}\}\cup\{11\bm{s}; \bm{s}\in \Bin_{n-2}\}$.\\
From this decomposition the sequence of vertices of $\overline{\Gamma}_n$ follows the induction
$$|V(\overline{\Gamma}_{n})|=|V(\overline{\Gamma}_{n-1})|+|V(\overline{\Gamma}_{n-2})|+2^{n-2}.$$
We deduce also a partition of the edges $\overline{\Gamma}_n$ in six sets:\\
-$|E(\overline{\Gamma}_{n-1})|$ edges between vertices of $\{0\bm{s}; \bm{s}\in \overline{\Fib}_{n-1}\}$.\\
-$|E(\overline{\Gamma}_{n-2})|$ edges between vertices of $\{10\bm{s}; \bm{s}\in \overline{\Fib}_{n-2}\}$.\\
-$|E(Q_{n-2}|$ edges between vertices of $\{11\bm{s}; \bm{s}\in \Bin_{n-2}\}$.\\
-Edges between vertices of $\{0\bm{s}; \bm{s}\in \overline{\Fib}_{n-1}\}$ 
and $\{10\bm{s}; \bm{s}\in \overline{\Fib}_{n-2}\}$.
Those edges are the $|V(\overline{\Gamma}_{n-2})|$ edges $\{00s,10s\}$ where $s\in \overline{\Fib}_{n-2}$.\\
-Edges between vertices of  $\{10\bm{s}; \bm{s}\in \overline{\Fib}_{n-2}\}$ and $\{11\bm{s}; \bm{s}\in \Bin_{n-2}\}$.
Those edges are the $|V(\overline{\Gamma}_{n-2})|$ edges $\{10s,11s\}$ where $s\in \overline{\Fib}_{n-2}$.\\
-Edges between vertices of  $\{0\bm{s}; \bm{s}\in \overline{\Fib}_{n-1}\}$ and $\{11\bm{s}; \bm{s}\in \Bin_{n-2}\}$.
Those edges are the $2^{n-2}-F_{n-1}$ edges $\{0s,1s\}$ where $s$ is a string of  $\overline{\Fib}_{n-1}^{1.}$ .\\
Therefore $|E(\overline{\Gamma}_n)|=|E(\overline{\Gamma}_{n-1})|+|E(\overline{\Gamma}_{n-2})|+(n-2)2^{n-3}+2(2^{n-2}-F_{n})+2^{n-2}-F_{n-1}$.
\end{proof}\qed 

We will call \emph{block} of a binary string $s$ a maximal substring of consecutive $1$'s.  Therefore a string in $\overline{\Fib}_n$ is as string with a least one block of length greater that 1.
The degree of a vertex of  $\Gamma_{n}$ lies between $\lfloor(n+2)/3\rfloor$ and $n$. The  number of vertices of a given degree is determined in  \cite{klmope2011}.\\
We will now give a similar result for $\overline{\Gamma}_{n}$.

\begin{theorem}
The degree of a vertex in $\overline{\Gamma}_{n}$ is $n$, $n-1$ or $n-2$ and the number of vertices of a given degree are:\\
$|E(\Gamma_{n-1})|$ vertices of degree $n-2$\\
$|E(\Gamma_{n-2})|$ vertices of degree $n-1$\\
$\sum_{k=0}^{n-4}{2^k|E(\Gamma_{n-k-3})|}$ vertices of degree $n$.
\end{theorem}
\begin{remark}
Using Proposition \ref{nbedges} these  numbers can be rewritten as, respectively,\\
$\frac{(n-1)F_{n}+(2n)F_{n-1}}{5}$, $\frac{(n-2)F_{n-1}+(2n-2)F_{n-2}}{5}$ and $2^{n}-\frac{(3n+7)F_{n}+(n+5)F_{n-1}}{5}$.
\end{remark}
\begin{proof} This is true for $n\leq3$ thus assume $n\geq 4$. Let $x$ be a vertex of $\overline{\Gamma}_{n}$ and consider the indices $i_l$ and $i_r$ such that $x_{i_l}x_{i_l+1}$ and $x_{i_r}x_{i_r+1}$ are the leftmost, rightmost respectively, pairs of consecutive $1$'s. Thus  $i_l=min\{i;x_ix_{i+1}=11\}$ and $i_r=max\{i;x_ix_{i+1}=11\}$. Consider the three possible cases\\
\begin{itemize}
\item $i_r=i_l$. Then there exists a unique block of length at least 2 and this block $x_{i_l}x_{i_l+1}$ is of length 2. Thus  $x_1\dots x_{i_l-1}\in \Fib_{i_l-1}^{.0}$ and $x_{i_l+2}\dots x_n\in \Fib_{n-i_l-1}^{0.}$.

For $i=i_l$ or $i=i_l+1$ the string $x+\delta_i$ is a Fibonacci string. For $i$ distinct of $i_l$ and $i_l+1$ then $x+\delta_i$ is a string  of $\overline{\Fib}_{n}$. Therefore $d(x)=n-2$.

Since $x_1\dots x_{i_l-1}$ and $x_{i_l+2}\dots x_n$ are arbitrary strings of $\Fib_{i_l-1}^{.0}$ and $\Fib_{n-i_l-1}^{0.}$ the number of vertices of this kind is by Proposition \ref{nbstring} $\sum_{i_l=1}^{n-1}{F_{i_l}F_{n-i_l}}=|E(\Gamma_{n-1})|$.

\item $i_r=i_l+1$. Then there exists a unique block of length at least 2 and this block $x_{i_l}x_{i_l+1}x_{i_l+2}$ is of length 3. Thus $x= x_1\dots x_{i_l-1}111x_{i_l+3}\dots x_n$ where  $x_1\dots x_{i_l-1}\in\Fib_{i_l-1}^{.0}$ and  $x_{i_l+3}\dots x_n\in\Fib_{n-i_l-2}^{0.}$.

For $i=i_l+1$ the string $x+\delta_i$ is a Fibonacci string. For $i$ distinct of $i_l+1$  then $x+\delta_i$ is a string  of $\overline{\Fib}_{n}$. Therefore $d(x)=n-1$.

Since $x_1\dots x_{i_l-1}$ and $x_{i_l+3}\dots x_n$ are arbitrary strings of $\Fib_{i_l-1}^{.0}$ and $\Fib_{n-i_l-2}^{0.}$ the number of vertices of this kind is  $\sum_{i_l=1}^{n-2}{F_{i_l}F_{n-i_l-1}}=|E(\Gamma_{n-2})|$.

\item $i_r\geq i_l+2$. Then there exists a unique block of length at least 4  or there exist at least  two blocks of length at least 2.
In both cases for any $i\in[1,n]$ $x+\delta_i$ is a string  of $\overline{\Fib}_{n}$. Therefore $d(x)=n$.

Let $k=i_r-i_l-2$. Note that $k\in[0,n-4]$ and $k$ fixed $i_l\in[1,n-k-3]$. The strings $x_1x_2\dots x_{i_l-1}$ and $x_{i_l+k+4}x_{i_l+k+5}\dots x_n$ are arbitrary strings in  $\Fib_{i_l-1}^{.0}$ and $\Fib_{n-k-i_l-3}^{0.}$. Since $x_{i_l+2}\dots x_{i_r-1}$ is an arbitrary string in $\Bin_{k}$ the number of vertices of this kind is  $\sum_{k=0}^{n-4}{\sum_{i_l=1}^{n-k-3}{2^k F_{i_l}F_{n-k-i_l-2}}}=\sum_{k=0}^{n-4}{2^k|E(\Gamma_{n-k-3})|}$.
\end{itemize}
\end{proof}\qed 

The sequence $0,0,0,1,4,13,36,\dots$ formed by the numbers of vertices on degree n in $\overline{\Gamma}_n, (n\geq1)$ already appears in OEIS \cite{oeis} as sequence A235996 of the number of length n binary words that contain at least one pair of consecutive 0's followed by (at some point in the word) at least one pair of consecutive 1's. This is clearly the same sequence.

As noticed in Figure \ref{F1} $\overline{\Gamma}_{3}$ and $\overline{\Gamma}_{4}$ are isomorphic to $\Gamma_2$ and $\Gamma_4$ respectively. Our last result  complete this observation.
\begin{theorem}
For $n\geq 4$ $\Gamma_n$ is isomorphic to an induced subgraph of $\overline{\Gamma}_{n}$.
\end{theorem}
\begin{proof}
Let $n\geq4$ and define a mapping between binary strings of length $n$ by $\theta(x)=\theta(x_1x_2\dots x_n)=\overline{x}_4\overline{x}_2\overline{x}_3\overline{x}_1\overline{x}_5\overline{x}_6\dots \overline{x}_n$. Let $\sigma$  be the permutation on $\{1,2,\dots,n\}$ define by $\sigma(1)=4$, $\sigma(4)=1 $ and $\sigma(i)=i $ for $i\notin\{1,4\}$. 

Note first that $x\in\Fib_n$ implies $\theta(x)\in\overline{\Fib}_n$. Indeed since $x_2x_3\neq 11$
we have three cases
\begin{itemize}
\item $x_2x_3=00$ then  $\overline{x}_2\overline{x}_3=11$ is a substring of $\theta(x)$
\item $x_2x_3=10$ then $x_1=0$ and $\overline{x}_3\overline{x}_1=11$ is a substring of $\theta(x)$
\item $x_2x_3=01$ then $x_4=0$ and  $\overline{x}_4\overline{x}_2=11$ is a substring of $\theta(x)$.
\end{itemize}
 Therefore $\theta$ maps vertices of $\Gamma_n$ to vertices in $\overline{\Gamma}_n$.

Let $\{x,x+\delta_i\}$ be an edge of $\Gamma_n$ then by construction we have $\theta(x+\delta_{i})=\theta(x)+\delta_{\sigma(i)}$
and therefore $\theta(x)$ and $\theta(x+\delta_i)$ are adjacent in $\overline{\Gamma}_n$.

Since $\theta$ is a transposition we have also for all $i\in[1,n]$ that $\theta(x)+\delta_i=\theta(x+\delta_{\sigma(i)})$. Therefore  if  $\{\theta(x),\theta (y)\}$ is an edge in the subgraph induced by $\theta(\Gamma_n)$ then $\theta (y)=\theta(x)+\delta_i$ for some $i$, $y=x+\delta_{\sigma(i)}$ and thus $\{x,y\}\in E (\Gamma_n)$.
\end{proof} 
\qed

Since $0^n$ is a vertex of degree $n$ in $\Gamma_n$ this graph cannot be a subgraph of $\overline{\Gamma}_m$ for $m<n$ thus this mapping in  $\overline{\Gamma}_n$ is optimal. Conversely it might be interesting to determine the minimum $m$ such that $\overline{\Gamma}_{n}$ is isomorphic to an induced subgraph of $\Gamma_{m}$. We already know that $m\leq 2n-1$ since the hypercube $Q_n$ is an induced subgraph of  $\Gamma_{2n-1}$ \cite{km2012}. 
\bibliographystyle{plain}
\bibliography{edgesfibonaccicubescomplements}

\end{document}